\documentclass[sn-mathphys]{sn-jnl}
\usepackage[all]{xy}

\jyear{2022}%

\theoremstyle{thmstyleone}%
\newtheorem{theorem}{Theorem}[section]
\newtheorem{theorema}{Theorem}

\newtheorem{lemma}[theorem]{Lemma}%

\theoremstyle{thmstyletwo}%
\newtheorem{remark}{Remark}%

\theoremstyle{thmstylethree}%
\newtheorem{notation}{Notation}
\newcommand\elnode[1]{\makebox[0.5mm][c]{$\underset{#1}{\bullet}$}}
\newcommand\elnodeup[1]{\makebox[0.5mm][c]{$\overset{#1}{\bullet}$}}

\DeclareMathOperator{\Alt}{Alt}
\DeclareMathOperator{\Aut}{Aut}

\DeclareMathOperator{\ele}{l}

\DeclareMathOperator{\GL}{GL}

\DeclareMathOperator{\maxi}{m}

\DeclareMathOperator{\Norm}{N}

\DeclareMathOperator{\OLandau}{O}

\DeclareMathOperator{\Orth}{O}
\DeclareMathOperator{\Out}{Out}
\DeclareMathOperator{\PSL}{PSL}
\DeclareMathOperator{\PSp}{PSp}
\DeclareMathOperator{\PSU}{PSU}

\DeclareMathOperator{\SL}{SL}
\DeclareMathOperator{\Soc}{Soc}

\DeclareMathOperator{\SU}{SU}
\DeclareMathOperator{\Sym}{Sym}

\newcommand\F{\mathbb{F}}

\begin{document}

\title[Maximal subgroups of small index of almost simple groups]{Maximal subgroups of small index of finite almost simple groups}


\author[1]{\fnm{Adolfo} \sur{Ballester-Bolinches}}\email{Adolfo.Ballester@uv.es}
\equalcont{These authors contributed equally to this work.}

\author*[2]{\fnm{Ram\'on} \sur{Esteban-Romero}}\email{Ramon.Esteban@uv.es}
\equalcont{These authors contributed equally to this work.}

\author[3]{\fnm{Paz} \sur{Jim\'enez-Seral}}\email{paz@unizar.es}
\equalcont{These authors contributed equally to this work.}

\affil[1]{\orgdiv{Departament de Matem\`atiques}, \orgname{Universitat de Val\`encia}, \orgaddress{\street{Dr.\ Moliner, 50}, \city{Burjassot}, \postcode{46100}, \state{Val\`encia}, \country{Spain}}}

\affil*[2]{\orgdiv{Departament de Matem\`atiques}, \orgname{Universitat de Val\`encia}, \orgaddress{\street{Dr.\ Moliner, 50}, \city{Burjassot}, \postcode{46100}, \state{Val\`encia}, \country{Spain}}}

\affil[3]{\orgdiv{Departamento de Matem\'aticas}, \orgname{Universidad de Zaragoza}, \orgaddress{\street{Pedro Cerbuna, 12}, \city{Zaragoza}, \postcode{50009}, \state{Zaragoza}, \country{Spain}}}


\abstract{We prove in this paper that every almost simple
group $R$ with socle isomorphic to a simple group $S$ possesses a
conjugacy class of core-free maximal subgroups whose index coincides with the
smallest index $\ele(S)$ of a maximal subgroup of $S$ or a conjugacy class
of core-free maximal subgroups with a fixed index $v_S\le \ele(S)^2$,
depending only on~$S$. We also prove that the number of subgroups of the outer automorphism group of $S$ is bounded by $\log^3{\ele(S)}$ and $\ele(S)^2< \lvert S\rvert$.}



\maketitle

\section{Introduction}
All groups considered in this paper will be finite.

Given a group $G$, %
one
can ask how many elements one should choose uniformly and
at random to generate $G$ with a certain given probability. The fact that an ordered $r$-tuple $(g_1,\dots, g_r)$ generates $G$ is equivalent to the fact that $\{g_1,\dots, g_r\}$ is not contained in any maximal subgroup $M$ of~$G$. The probability that $\{g_1,\dots, g_r\}$ is contained in a maximal subgroup $M$ of~$G$ is $1/\lvert G:M\rvert^r$. Consequently, it is of relevant interest to find good bounds for the number $ \maxi_n(G)$ of maximal subgroups of a group $G$ of  a given index $n$. 

Note that if $M$ is a maximal subgroup of $G$, then $G/M_G$, where $M_G$ denotes the core of $M$ in~$G$, is a primitive group. Consequently, the proof of many results in this field relies on the subgroup structure of such groups. 

According to the theorem of Baer \cite{Baer57} (see also \cite[Theorem~1.1.7]{BallesterEzquerro06}), there are three types of primitive groups, according to whether they have a unique abelian minimal normal subgroup (type~1), a unique non-abelian minimal normal subgroup (type~2), or two non-abelian minimal normal subgroups (type~3). The theorem of O'Nan and Scott (see
 \cite[Theorem~1.1.52]{BallesterEzquerro06}) describes the different possibilities for a primitive pair $(G, U)$ composed of a primitive group $G$ of type~2 and a core-free maximal subgroup $U$ of $G$. In all cases,  the corresponding primitive pair is related to a primitive pair corresponding to an almost simple group with socle $S$, where the minimal normal subgroup of~$G$ is a direct product of copies of~$S$. This makes crucial the study of the indices of core-free maximal subgroups of almost simple groups in the study of core-free maximal subgroups of primitive groups of type~2.

\begin{notation}
  We denote by $\ele(X)$ the least degree of a faithful transitive
permutation representation of a group~$X$, that is, the smallest index
of a core-free subgroup of~$G$. 
\end{notation}

The aim of this paper is to prove that every almost simple
group $R$ with socle isomorphic to a simple group $S$ possesses a
conjugacy class of core-free maximal subgroups whose index coincides with the
smallest index $\ele(S)$ of a maximal subgroup of $S$ or a conjugacy class
of core-free maximal subgroups with a fixed index $v_S\le \ele(S)^2$,
depending only on~$S$. We also prove that the number of subgroups of the outer automorphism group of $S$ is bounded by $\log^3{\ele(S)}$ and that $\ele(S)^2< \lvert S\rvert$.

These results will be applied in~\cite{BallesterEstebanJimenez-maximal-bounds-appl} to obtain lower bounds for the number of elements needed to generate a group with a certain probability and to obtain good lower bounds for the number of maximal subgroups of a given index of a group. They are also useful to estimate the number of
possible socles of primitive groups of type~2 with a core-free maximal
subgroup of a given index.

Our first main result includes relevant information over the smallest
index ${\ele(S)}$ of a maximal subgroup of a non-abelian simple group $S$
and shows the existence of relevant subgroups of small index in an
almost simple group with socle $S$. Moreover, we see that the order of the outer automorphism group of $S$ is bounded by
$3\log \lvert S\rvert$. Here we reserve the symbol $\log$ to denote the logarithm to the base~$2$. This last bound clarifies and improves the one used in the proof of
\cite[Lemma~2.3]{BorovikPyberShalev96}, $\lvert \Out S\rvert \le
\OLandau(\log^2n)$, and, as we will show in Remark~\ref{remark-PSL-log}, this
bound is best possible. The bound $\lvert {\Out S}\rvert \le 3\log
\ele(S)$ also appears in
\cite[Lemma~7.7]{GuralnickMarotiPyber17}, we
present here a proof for completeness.

We say that a maximal subgroup of a
simple group $S$ is \emph{ordinary} if its conjugacy class in $S$ coincides
with its conjugacy class in $\Aut(S)$. 
Most simple groups $S$ possess a conjugacy class of ordinary maximal subgroups of the smallest possible index $\ele(S)$ and, by Lemma~\ref{lemma-conj-char-class} below, every almost simple group with socle $S$ possesses a  maximal subgroup of index $\ele(S)$. However, some simple groups do not have ordinary maximal subgroups of the smallest possible index. These groups constitute the classes $\mathfrak{X}$ and $\mathfrak{Y}$ that we define below.

\begin{notation}
Let $\mathfrak{X}$ be the class of simple groups composed of the
following groups: 
\begin{enumerate}
\item the linear groups $\PSL_3(q)$, where $q=p^f> 3$ is a power of
  a prime $p$ with $f$
odd;
\item the linear groups $\PSL_n(q)$, with $q$ a prime power and $n=5$ or $n\ge 7$;
\item the symplectic groups $\PSp_4(2^f)$, $f\ge 2$.
\end{enumerate}
\end{notation}

\begin{notation}
Let $\mathfrak{Y}$ be the class of simple groups composed of the
following groups: 
\begin{enumerate}
\item the Mathieu
group $\mathrm{M}_{12}$;
\item the O'Nan group $\mathrm{O'N}$;
\item the Tits
group ${}^2\!F_4(2)'$;
\item the linear groups
$\PSL_2(7)\cong
\PSL_3(2)$, $\PSL_2(9)\cong \Alt(6)$, $\PSL_2(11)$, $\PSL_3(3)$;
\item the linear groups
$\PSL_3(q_0^2)$, with $q_0$ a prime power;
\item the linear groups $\PSL_4(q)$, with $q$ a prime
power, $q>2$;
\item the linear groups $\PSL_6(q)$, with $q$ a prime power;
\item the unitary group $\PSU_3(5)$;
\item the orthogonal groups $\Orth_8^+(q)$, with $q$ a prime power;
\item the orthogonal groups $\Orth_n^+(3)$, with $n\ge 10$;
\item the exceptional groups of Lie type $G_2(3^f)$, with
$f\ge 1$;
\item the exceptional groups of Lie type $F_4(2^f)$, with $f\ge 1$;
\item the exceptional groups of Lie type $E_6(q)$, with $q$ a prime power.
\end{enumerate}
\end{notation}

In the simple groups $S$ of the class~$\mathfrak{Y}$, there are no ordinary maximal subgroups of index $\ele(S)$, but we can find a number $v_S\le \ele(S)^2$ that depends only on $S$ such that $S$ has a conjugacy class of ordinary maximal subgroups of index $v_S$. Again by Lemma~\ref{lemma-conj-char-class}, every almost simple group with socle $S$ possesses a conjugacy class of maximal subgroup of this index $v_S$. In other words, $v_S$ appears as a common index of a core-free maximal subgroup for all almost simple groups with socle $S$. The class~$\mathfrak{X}$ is composed of the rest of the simple groups, that is, all simple groups that do not have a conjugacy class of ordinary maximal subgroups with index bounded by $\ele(S)^2$. However, we will prove that in the groups of the class~$\mathfrak{X}$, we can find a number $v_S\le \ele(S)^2$ such that every almost simple group $R$ with socle $S$ possesses a maximal subgroup of order $\ele(S)$ or a maximal subgroup of index~$v_S$. We present this in detail in  Theorem~\ref{th-simple}.

\begin{theorema}\label{th-simple}
  Let $S$ be a simple group.
  \begin{enumerate}
  \item If $S$ does not belong to $\mathfrak{X}\cup\mathfrak{Y}$, then $S$ has a
    conjugacy class of ordinary maximal subgroups. In particular, if $R$ is an almost simple group with $S\leq R \leq \Aut(S)$, then $R$ has a conjugacy class of core-free maximal subgroups of index $\ele(S)$.
  \item If $S$ belongs to $\mathfrak{Y}$, then $S$ has at least two
    conjugacy classes of maximal subgroups of the smallest index
    ${\ele(S)}$ and there exists a number $v_S\le {\ele(S)}^2$, depending only
    on $S$, such that if $R$ is an almost simple group with $S\leq R\leq \Aut(S)$, then $R$ has a conjugacy class of core-free maximal
    subgroups with index $v_S$. 
  \item If $S$ belongs to $\mathfrak{X}$, then $S$ has at least two
    conjugacy classes of maximal subgroups of the smallest index ${\ele(S)}$
    and there exists a number $v_S\le {\ele(S)}^2$, depending only on~$S$,
    such that if $R$ is an almost simple group with $S\leq R\leq \Aut(S)$, then $R$
    has at least two conjugacy classes of core-free maximal subgroups with index
    ${\ele(S)}$ or one conjugacy class of core-free maximal subgroups with
    index $v_S$.
  \item In all cases, ${\ele(S)}^2 < \lvert S\rvert$\label{en-lemma-simple-4}  and
     $\lvert{\Out S}\rvert \le 3\log {\ele(S)}$.
   \item  If, in addition,\label{en-lemma-simple-1log}
     \begin{enumerate}
     \item $S\not\cong \Alt(6)$;
     \item $S$
       is not of the
       form $\PSL_n(q)$ with $q=p^f$ and
       \begin{enumerate}
       \item $n\ge 3$, $p\in\{2,3,5,7\}$, and $\gcd(n, q-1)>1$,
         or
       \item  $n=2$ and $q=3^f$;
       \end{enumerate}
     \item 
       $S$ is not of the form $\PSU_n(q)$ with
       $q=p^f$ and
       \begin{enumerate}
       \item $n=3$ and $p=3$, or
       \item  $n=3$ and $q=5$, or  
       \item $n\ge 4$, $p=2$, $f>1$ and $\gcd(n, q+1)>1$, or
       \item $p=3$, $n=5$, and
       \end{enumerate}
       
     \item $S\not \cong {\Orth_8^+(q)}$ with $q=p^f$ and
       $p\in\{3,5,7,11,13\}$, 
     \end{enumerate}
   \end{enumerate}
   then $\lvert \Out S\rvert \le \log
     {\ele(S)}$.
\end{theorema}

\begin{remark}\label{remark-ONan}
  According to \cite{Atlas85}, the automorphism group of the O'Nan simple group
  $S\cong{\operatorname{O'N}}$ has all core-free maximal subgroups of
  index greater than its order, so
  Theorem~\ref{th-simple}~(\ref{en-lemma-simple-4}) cannot be
  extended to the core-free maximal subgroups of almost simple groups.
\end{remark}

\begin{theorema}\label{th-s-log3}
  The number of subgroups of the outer automorphism group of a non-abelian
  simple group $S$ is bounded by $\log^3{\ele(S)}$.
\end{theorema}

Unless otherwise stated, we will follow the notation of the books
\cite{DoerkHawkes92} and \cite{BallesterEzquerro06}. Detailed
information about primitive groups and  chief factors of a group can be found in \cite[Chapter~1]{BallesterEzquerro06}.
\section{Proofs}\label{sect-max-small-index}

\begin{figure}[tb]
  \SelectTips{cm}{}

\centering

  \[\xymatrix@1@R-6mm{A_l\colon&{\elnode{1}}\ar@{-}[r]&{\elnode{2}}\ar@{-}[r]&{\elnode{3}}\ar@{.}[r]&{\elnode{l-3}}\ar@{-}[r]&{\elnode{l-2}}\ar@{-}[r]&{\elnode{l-1}}\ar@{-}[r]&{\elnode{l}}\\
B_l\colon&{\elnode{1}}\ar@{-}[r]&{\elnode{2}}\ar@{-}[r]&{\elnode{3}}\ar@{.}[r]&{\elnode{l-3}}\ar@{-}[r]&{\elnode{l-2}}\ar@{-}[r]&{\elnode{l-1}}\ar@{=>}[r]&{\elnode{l}}\\
C_l\colon&{\elnode{1}}\ar@{-}[r]&{\elnode{2}}\ar@{-}[r]&{\elnode{3}}\ar@{.}[r]&{\elnode{l-3}}\ar@{-}[r]&{\elnode{l-2}}\ar@{-}[r]&{\elnode{l-1}}\ar@{<=}[r]&{\elnode{l}}\\
&&&&&&{\elnode{\quad l-1}}\\
    D_l\colon&{\elnode{1}}\ar@{-}[r]&
    {\elnode{2}}\ar@{-}[r]&{\elnode{3}}\ar@{.}[r]&{\elnode{l-3}}\ar@{-}[r]&{\elnode{l-2\quad
      }}\ar@{-}[ur]\ar@{-}[dr]\\
    &&&&&&{\elnode{l}}\\
E_6\colon&{\elnode{1}}\ar@{-}[r]&{\elnode{3}}\ar@{-}[r]&{\elnodeup{4}}\ar@{-}[d]\ar@{-}[r]&{\elnode{5}}\ar@{-}[r]&{\elnode{6}}\\
  &&&{\elnode{2}}\\
  E_7\colon&{\elnode{1}}\ar@{-}[r]&{\elnode{3}}\ar@{-}[r]&{\elnodeup{4}}\ar@{-}[d]\ar@{-}[r]&{\elnode{5}}\ar@{-}[r]&{\elnode{6}}\ar@{-}[r]&{\elnode{7}}\\
  &&&{\elnode{2}}\\
  E_8\colon&{\elnode{1}}\ar@{-}[r]&{\elnode{3}}\ar@{-}[r]&{\elnodeup{4}}\ar@{-}[d]\ar@{-}[r]&{\elnode{5}}\ar@{-}[r]&{\elnode{6}}\ar@{-}[r]&{\elnode{7}}\ar@{-}[r]&{\elnode{8}}\\
  &&&{\elnode{2}}\\F_4\colon&{\elnode{1}}\ar@{-}[r]&{\elnode{2}}\ar@{=>}[r]&{\elnode{3}}\ar@{-}[r]&{\elnode{4}}\\
  G_2\colon&{\elnode{1}}\ar@3{<-}[r]&{\elnode{2}}}\]
\caption{Dynkin diagrams for the simple groups of Lie type}
  \label{fig-Dynkin}
\end{figure}
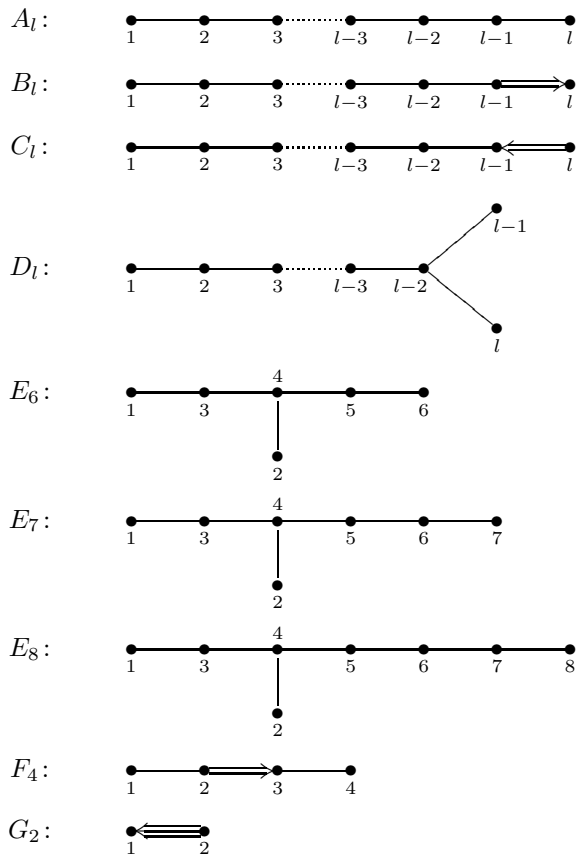

Our results will depend heavily on the classification of simple
groups. For the simple groups of Lie type, we number the nodes of the
corresponding
Dynkin diagrams as in Figure~\ref{fig-Dynkin} and denote accordingly
the associated parabolic subgroups. 
The values of ${\ele(S)}$ for the
simple groups of Lie type have been
computed in the series of papers of Mazurov \cite{Mazurov93}, Vasil'ev
and Mazurov \cite{VasilevMazurov94}, and Vasilyev
\cite{Vasilyev96,Vasilyev97,Vasilyev98}. 

\begin{lemma}\label{lemma-conj-char-class}
  Suppose that $S$ is a non-abelian simple group and let 
  $A=\Aut(S)$. We identify $S$ with the subgroup of $A$ composed by
  all inner automorphisms induced by~$S$. If $M$ is a maximal subgroup of $S$
  such that the conjugacy class of $M$ in $S$ is invariant under the
  action of $A$, then $W=\Norm_A(M)$ is a maximal subgroup of
  $A$, $W\cap S=M$, and $\lvert S:M\rvert =\lvert A:W\rvert$.
\end{lemma}
\begin{proof}
  Note that the length of the conjugacy class of $M$ in $S$ is equal
  to the length of the conjugacy class of $M$ in $A$. Since this
  length coincides with the index of the corresponding normalisers, we
  have that $\lvert S:\Norm_S(M)\rvert=\lvert A:
  \Norm_A(M)\rvert$. In particular, if $S\le T\le A$, then
  the conjugacy class of $M$ in $T$ coincides with the conjugacy class
  of $M$ in $S$. Since the length of the conjugacy
  class coincides with the index of the normaliser and $M=\Norm_S(M)$,
  we have that
  \begin{equation}
    \lvert S:M\rvert=\lvert T:{\Norm_T(M)}\rvert=\lvert
    A:{\Norm_A(M)}\rvert
    \label{eq-intermediate}
  \end{equation}
  for every $T$ with $S\le T\le A$. Let
  $W=\Norm_A(M)$. Then $W\cap S=\Norm_S(M)=M$ since $M$ is a maximal
  subgroup and $S$ is a non-abelian simple group. We prove now that
  $W$ is a maximal subgroup of $A$. Suppose that $W\le V\le A$. By
  taking intersections with $S$, we obtain that $M\le V\cap S\le
  S$. Since $M$ is maximal in $S$, $M=V\cap S$ or $V\cap S=S$. In the
  first case, since $M=V\cap S$ is a normal subgroup of $V$, we obtain
  that $V\le \Norm_A(M)=W$ and so $V=W$. In the second case, $S\le
  V$. As $\Norm_A(M)=\Norm_V(M)$,
  by~\eqref{eq-intermediate}, it turns out that $V=A$.
\end{proof}

By Lemma~\ref{lemma-conj-char-class}, if $X$ is an almost simple group
with $\Soc(X)\cong S$ and $M$ is an ordinary maximal subgroup of $S$,
then $\Norm_X(M)$ is a maximal subgroup of $X$ of index $\lvert
S:M\rvert$. We will use this fact without mentioning it explicitly.

\begin{proof}[Proof of Theorem~\ref{th-simple}]
  We will analyse the different possibilities for $S$ in the classification of
  finite simple groups. We note that the condition ${\ele(S)}^2<\lvert
  S\rvert$ is equivalent to affirming that there is a maximal subgroup
  of $S$
  with index less than its order. We warn the reader that the
  information about the maximal subgroups comes from several sources
  and, in order to make it easier to check the results, we have
  preferred to adhere to the notation of the corresponding source,
  even if in some cases there appear some inconsistencies in the notation.
  \newcommand{\smalltitle}[1]{\par\medskip\emph{#1}\nopagebreak\par\smallskip}
  \smalltitle{Sporadic simple groups}

  Suppose first that $S$ is a sporadic simple group. It is clear that
  if the outer automorphism group of $S$ is trivial, then
  $S\notin\mathfrak{X}\cup \mathfrak{Y}$ and  the result
  is trivially valid. In the other cases, the outer automorphism group
  has order~$2$. In the sporadic simple groups $\mathrm{M}_{22}$, $\mathrm{J}_2$,
  $\mathrm{Suz}$, $\mathrm{HS}$, $\mathrm{McL}$, $\mathrm{He}$,
  $\mathrm{Fi}_{22}$, $\mathrm{HN}$, and $\mathrm{J_3}$, according to
  the Atlas \cite{Atlas85}, the largest maximal subgroups are ordinary
  and so ${\ele^*(A)}={\ele(S)}$. The maximal subgroups of the group
  $\mathrm{Fi}_{24}'$ and its automorphism group $\mathrm{Fi}_{24}$
  have been studied in \cite{LintonFischer91}. The smallest index
  maximal subgroup $\mathrm{Fi}_{23}$ is ordinary. In the Mathieu
  group $\mathrm{M}_{12}$, there are two classes of the smallest index
  maximal subgroup, with index~$12$ and there is a class of ordinary
  maximal subgroups of index $144=12^2$ %
  (see
  \cite{Atlas85}). In the
  O'Nan group $S\cong \mathrm{O'N}$, according again to \cite{Atlas85}, there
  are two conjugacy classes of
  maximal subgroups of type $\mathrm{L}_3(7):2$, of the smallest index
  $122\,760$, fused under the outer automorphism, giving a conjugacy
  class of novel
  maximal subgroups of type $7^{1+2}_+:(3\times D_{16})$ and index
  $55\,978\,560\le {\ele(S)}^2$.

  For all the sporadic groups $S$, we also see that the inequality
  ${\ele(S)}^2<\lvert S\rvert$ holds for all groups whose maximal subgroups
  have been described in \cite{Atlas85}. The conclusion for the Baby
  Monster group $B$ follows, by \cite{Wilson99}, for the smallest index
  maximal subgroup $2.{}^2\!E_6(2):2$. The conclusion for the Monster
  group $M$ is also true, by \cite{Wilson10-Moonshine,NortonWilson13},
  with the maximal subgroup $2.B$.  Finally, it is clear that $\lvert \Out
  S \rvert \le 2\le \log {\ele(S)}$ for all sporadic simple groups $S$.

  According to \cite{Atlas85}, the Tits
  group ${}^2\!F_4(2)'$ has an outer automorphism group of order~$2$
  and two conjugacy classes of subgroups of type $\PSL_3(3):2$ of the
  smallest possible
  index, $1\,600$, fused under the graph automorphism, and an ordinary
  maximal subgroup of type $2.[2^8]:5:4$ of index $1\,755\le
  1\,600^2$. Furthermore, ${\ele(S)}^2<\lvert S\rvert$ and $\lvert \Out
  S\rvert \le 2\le \log {\ele(S)}$.

  \smalltitle{Alternating groups}

  Suppose now that $S\cong {\Alt(n)}$ with $n\ge 7$ or $n=5$. Then
  $\Alt({n-1})$ is an ordinary maximal subgroup of $S$ of the smallest
  possible order, $n$. If $S\cong \Alt(6)\cong \PSL_2(9)$, we see in
  \cite{Atlas85} that the smallest index of a maximal subgroup of $S$
  is $6$, while $S$ possesses an ordinary maximal subgroup of type
  $3^2:4$ and 
  index~$10<6^2$. Clearly, $\ele\bigl({\Alt(n)}\bigr)^2<\lvert {\Alt(n)}\rvert$ for all $n\ge
  5$ and $\lvert \Out \Alt(n)\rvert=2 \le \log {\ele(\Alt(n))}$ if $n\ne 6$, and
  $\lvert \Out \Alt(6)\rvert=4\le 3\log {\ele(\Alt(6))}$.

  \smalltitle{Linear groups}
  We start with the linear groups on dimension~$2$.
  Suppose first that $S\cong
  \PSL_2(q)$ with 
  $q\in \{5,7,8,9,11\}$. We have that $S=\PSL_2(5)\cong {\Alt(5)}$ and
  $S=\PSL_2(9)\cong {\Alt(6)}$ have been studied before. Moreover, according
  to \cite{Atlas85}, $S=\PSL_2(7)$
  has two conjugacy classes of maximal subgroups of index ${\ele(S)}=7$ and a
  class of ordinary maximal subgroups of index $8$, $S=\PSL_2(8)$
  has an ordinary maximal subgroup of index ${\ele(S)}=9$, and
  $S\cong\PSL_2(11)$ has two conjugacy classes of maximal subgroups of
  index~${\ele(S)}=11$ and another conjugacy class of ordinary maximal
  subgroups of index $12<{\ele(S)}^2$. 
  We can
  see in \cite{Atlas85}  the existence of maximal subgroups in $S$ and
  in all almost simple groups with the prescribed indices and 
  that
  $\lvert \Out \PSL(2,q)\rvert \le \log {\ele(S)}$ when $q\in\{5,7,8,11\}$ and that
  $\lvert \PSL(2,9)\rvert\le 3\log {\ele(S)}$.  Since
  the parabolic subgroups have index $q+1$ and order $q(q-1)/d$ with
  $d=\gcd(q-1,2)$, we see that ${\ele(S)}\le q+1\le q(q-1)/d$ and so
  ${\ele(S)}^2<\lvert S\rvert$. 

  Suppose that $S\cong \PSL_2(q)$, with $q=p^f\ge 13$. Then the
  parabolic (Borel)
  subgroups are  the smallest index maximal subgroups and are
  ordinary by \cite[Theorem~1]{Mazurov93}. Their index is ${\ele(S)}=q+1$
  and their order is $q(q-1)/d$. Then $q^2-3q-2=q(q-3)-2\ge 0$, which implies that $\ele(S)=q+1\le q(q-1)/2\le q(q-1)/d$ and so $\ele(S)^2\le \lvert S\rvert$.
  Furthermore,  
  $\lvert \Out S\rvert =\gcd(2,q-1)\cdot f\le 2f\le 2\log p^f\le 2\log
  {\ele(S)}$. If, in addition $p\ge 5$, then
  $2f\le \log p^f\le \log {\ele(S)}$, while if $p=2$, then $\lvert \Out
  S\rvert = 1\cdot f\cdot 1\le \log {\ele(S)}$.

  We consider now the linear groups on dimension greater than~$2$. The groups $\PSL_3(2)\cong \PSL_2(7)$ and $\PSL_4(2)\cong {\Alt(8)}$ have been
  considered before. 
  Let $R$ be an almost simple group with $S=\Soc(R)\cong \PSL_n(q)$, where $n\ge 3$, that is, $S\le R\le A$ with $A\cong \Aut(S)$, and
  suppose that $(n,q)\notin\{(3,2),(4,2)\}$. According to
  \cite[Theorem~1]{Mazurov93}, ${\ele(S)}=(q^n-1)/(q-1)$ is the index of a
  parabolic subgroup. 
  Then
  the outer automorphism group of $S$ is isomorphic to $[C_d][C_f]C_2=\langle \delta, \phi, \gamma\rangle$,
  with $q=p^f$, $p$ a prime, and $d=\gcd(n, q-1)$ (see, for example,
  \cite{Atlas85}).  
  If $R/S$ is contained in $\langle \delta, \phi\rangle$, the parabolic subgroups of type
  $P_1$, which are the stabilisers of a $1$-dimensional subspace, induce
  maximal subgroups of $R$ of index $(q^n-1)/(q-1)$ since they are stabilised by $\langle \delta, \phi\rangle$. If $R/S$ is not
  contained in $\langle \delta, \phi\rangle$, then the double parabolic subgroups
  $P_{1,n-1}$, stabilisers of pairs of subspaces $(W,U)$ with
  $W<U$ and $1=\dim W=n-\dim U$, induce maximal subgroups of $R$ of index
  $v_S=(q^n-1)(q^{n-1}-1)/(q-1)^2$ since they are stabilised by $\langle \delta,\phi,\gamma\rangle$. 
  Then $R$ has a maximal subgroup of index
  $\ele(S)=(q^n-1)/(q-1)$ or of index $v_S=(q^n-1)(q^{n-1}-1)/(q-1)^2\le {\ele(S)}^2$, according
  to whether or not $R/S$ is contained in $\langle \delta, \phi\rangle$, respectively. In
  particular, $\ele^*(R)\le {\ele(S)}^2$. It is clear that
  \[{\ele(S)}^2=\left(\frac{q^n-1}{q-1}\right)^2<\frac{q^{n(n-1)/2}(q^n-1)(q^{n-1}-1)\dotsm (q^2-1)}{\gcd(n, q-1)}=\lvert S\rvert.\]

  If
  $S\cong \PSL_4(2)$, then $\lvert\Out S\rvert=2<\log {\ele(S)}$.
  Furthermore, for $S=\PSL_n(q)$ with $n\ge 3$, $(n,q)\ne(4,2)$,  we have that
  ${\ele(S)}=(q^n-1)/(q-1)$ and so
  \begin{align*}\lvert
    \Out S\rvert &\le  n\cdot f\cdot 2 \le (3/2)(n-1)\cdot f\cdot 2\\
    &=3\log
    \bigl((2^f)^{n-1}\bigr)\le  3\log q^{n-1}< 3\log
    {\ele(S)}.
  \end{align*}
  If, in addition, $p\ge 11$, then $3\log (2^f)^{n-1}\le \log
  (p^f)^{n-1}\le \log {\ele(S)}$. Furthermore, if $p\ge 2$ and $\gcd(n, q-1)=1$, then $\lvert\Out S\rvert=2f\le q^2+1\le \ele(S)$.

  Now we analyse the cases for which $G\notin\mathfrak{X}$.
  Note  that, according to \cite{Atlas85}, the
  group $S\cong \PSL_3(3)$ has two conjugacy classes of maximal
  subgroups of index ${\ele(S)}=13$ and a conjugacy class of ordinary
  maximal subgroups of type $13:3$ and index $144<{\ele(S)}^2<\lvert S\rvert$  and  $S\cong\PSL_3(4)$ has two conjugacy classes of maximal
  subgroups of index ${\ele(S)}=21$ and a conjugacy class of ordinary
  maximal subgroups of type $3^2:Q_8$ and index $280<{\ele(S)}^2<\lvert S\rvert$.

    Suppose 
  that $S=\PSL_3(q)$ with $q=q_0^2$, where $q_0$ a prime power, $q_0>2$. We have that $\ele(S)=(q^3-1)(q-1)=(q_0^6-1)/(q_0^2-1)$. Then
  $\gcd(q_0+1,3)=1$ or $\gcd(q_0-1,3)=1$. By
  \cite[Table~8.3]{BrayHoltRoneyDougal13}, in the first case, $\SL_3(q)$ has an ordinary maximal subgroup isomorphic
  to $\SL_3(q_0)$ of index
  $q_0^3(q_0^3+1)(q_0^2+1)=q_0^8+q_0^6+q_0^5+q_0^3<(q_0^4+q_0^2+q_0)^2={\ele(S)}^2$, while,
  in the second case,
 $\SL_3(q)$ has
 an ordinary
 maximal subgroup isomorphic to $\SU_3(q_0)$, of index
  $q_0^3(q_0^3-1)(q_0^2+1)<\bigl((q_0^6-1)/(q_0^2-1)\bigr)^2={\ele(S)}^2$.
  
  For the group $S\cong \PSL_4(q)$, $q>2$, according to
  \cite[Table~8.8]{BrayHoltRoneyDougal13}, we have that $\SL_4(q)$ has two classes of maximal
  subgroups of the
  smallest index ${\ele(S)}=(q^4-1)/(q-1)$ and a conjugacy class of
  ordinary maximal subgroups of type $E_q^4:{\SL_2(q)}\times{\SL_2(q)}:(q-1)$ and
  index $v_S=(q^2+1)(q^3-1)/(q-1)$. Therefore
  $v_S/{\ele(S)}^2=(q-1)(q^3-1)/((q^2-1)(q^4-1))<1$ and so $v_S<{\ele(S)}^2$.
  
  For the group $S\cong \PSL_6(q)$, according to
  \cite[Table~8.24]{BrayHoltRoneyDougal13}, we have that $\SL_6(q)$ has two classes of maximal
  subgroups of the
  smallest index ${\ele(S)}=(q^6-1)/(q-1)$ and a conjugacy class of
  ordinary maximal subgroups of type $E_q^9:{\SL_3(q)}\times{\SL_3(q)}:(q-1)$ and
  index $v_S=(q^5-1)(q^4-1)(q^3+1)/((q-1)^2(q+1))$. Therefore
  $v_S/{\ele(S)}^2=(q^5-1)(q^2+1)(q^2-1)/((q^6-1)(q+1)^2(q^3-1))<1$ and so
  $v_S<{\ele(S)}^2$.


  \smalltitle{Symplectic groups}

  Now suppose that $S\cong \PSp_n(q)$ with $n\ge 4$ even. Then 
  \[\lvert
  S\rvert= \frac{1}{d}q^{(n/2)^2}\Biggl(\prod_{i=1}^{n/2}(q^{2i}-1)\Biggr),\]
  where $d=\gcd(2,q-1)$. The smallest
  index maximal subgroups of $S$ are described in
  \cite[Theorem~2]{Mazurov93}. If $(n,q)=(4,3)$
  then $S\cong \PSU_4(2)$ and $G$ has an ordinary maximal subgroup of
  type $2^4:{\Alt(5)}$, of index ${\ele(S)}=27$ and order~$960$.  In this case,
  $\lvert \Out S\rvert=2<\log {\ele(S)}$.

  Suppose that $n\ge 6$
  and $q=2$. Then $S\cong \PSp_n(2)\cong {\Orth_{n+1}(2)}$ and the smallest index
  maximal subgroup of $S$ is isomorphic to $\Orth_n^{-}(2)$ and has index
  $2^{n/2-1}(2^{n/2}-1)$. By \cite[Tables~8.28, 8.48, 8.64,
  8.80]{BrayHoltRoneyDougal13} and
  \cite[Table~3.5.C]{KleidmanLiebeck90}, this subgroup is
  ordinary.  It is clear that this index is smaller than the order of
  this subgroup, namely
  $2^{(n/2)(n/2-1)}(2^{n/2}+1)\prod_{i=1}^{n/2-1}(q^{2i}-1)$, and that
  $\lvert \Out S\rvert=1<\log {\ele(S)}$.

  If $n=4$ and $q=2^f$, then $\lvert S\rvert=q^4(q^4-1)(q^2-1)$ and there are two conjugacy classes of
  parabolic maximal subgroups of type $E_q^3:{\GL_2(q)}$ and index
  ${\ele(S)}=(q^4-1)/(q-1)$ fused under the
  graph automorphism, by
  \cite[Table~8.14]{BrayHoltRoneyDougal13}. There is a novelty
  subgroup $[q^4]:(C_{q-1})^2$, maximal under subgroups not contained
  in the subgroup $\langle \phi\rangle$ generated by the field automorphism, of index
  \[\frac{(q^4-1)(q+1)}{q-1}<
  \left(\frac{q^4-1}{q-1}\right)^2={\ele(S)}^2.\]
  In this case, 
  \[\lvert\Out S\rvert = 1\cdot f\cdot 2=2f\le \log q^2\le \log
  {\ele(S)}.\]

  For the rest of the values of $(n,q)$, the smallest index
  maximal subgroup is a parabolic subgroup, which can be taken to have
  the form $[q^{n-1}]:((q-1).\PSp_{n-2}(q))$, with index
  ${\ele(S)}=(q^n-1)/(q-1)$
  and order 
  \[\frac{1}{d}q^{n-1}(q-1)q^{((n-2)/2)^2}\prod_{i=1}^{(n-2)/2}(q^{2i}-1)>q^n\ge
  \frac{q^n-1}{q-1}={\ele(S)}. \]
  Here  $\lvert\Out S\rvert=\gcd(2,q-1)\cdot f\cdot 1=2f$ if $p$ is odd
  and $\lvert \Out S\rvert=\gcd(2,q-1)\cdot f\cdot 1=f$ if $p=2$. In any
  case,
  \[\lvert \Out S\rvert \le 2f\le \log q^{n-1}\le \log {\ele(S)}.\]

  \smalltitle{Unitary groups}

  Suppose now that $S\cong \PSU_n(q)$ with $n\ge 3$ and $q>2$ if $n=3$; then
  $\lvert \Out S\rvert \le n\cdot f\cdot 1$ with $p^f=q^2$. The smallest
  index of a maximal subgroup of $S$ is given in
  \cite[Theorem~3]{Mazurov93}.
  
  We consider first the case $S\cong
  \PSU_3(5)$. We see in
  \cite{Atlas85} that the automorphism group of $S$ is isomorphic to
  $S_3$ and that there are three conjugacy classes of maximal
  subgroups of the smallest possible index, of type $\Alt(7)$ and index ${\ele(S)}=50$
  and order $2\,520$. Moreover, there is an ordinary maximal subgroup
  of type $5_+^{1+2}:8$ and index $126\le 50^2$.  In this case, $\lvert
  \Out S\rvert =3\cdot 2\cdot 1=6\le 3\log {\ele(S)}$.

  Suppose that $n=3$,
  $q\notin\{2,5\}$. Then $S$ has a conjugacy class of parabolic ordinary maximal
  subgroups of type $[q^3]:((q^2-1)/d)$, with $d=\gcd(3,q+1)$ and
  index ${\ele(S)}=q^3+1$, as we can see in
  \cite[Table~8.5]{BrayHoltRoneyDougal13}. Clearly, ${\ele(S)}^2<\lvert
  S\rvert$ and $\lvert \Out S\rvert \le \gcd(3,q-1)\cdot 2f\le \log
  2^{6f}$. If $p\ge 5$,
  $\log 2^{6f}\le \log q^3\le \log {\ele(S)}$. If $p=3$, then $\log
  2^{6f}\le \log 3^{4f}=(4/3)\log 3^{3f}\le (4/3)\log {\ele(S)}$.

  Now assume that $n=4$. Then
  $S$ has a conjugacy class
  of parabolic maximal subgroups of type
  $[q^4].{\SL_2(q^2)}:((q-1)/d)$, with $d=\gcd(q+1,4)$, whose index is
  ${\ele(S)}=(q^3+1)(q+1)$ and whose order is $q^6(q^4-1)(q-1)>{\ele(S)}$. These
  subgroups are ordinary, as shown in
  \cite[Table~8.10]{BrayHoltRoneyDougal13}. Moreover, $\lvert
  \Out S\rvert =\gcd(4,q+1)\cdot 2f\cdot 1 =8f\le  2\log 2^{4f}\le 2\log
  q^4\le 2\log {\ele(S)}$. If, in addition, $p\notin\{2,3\}$, then
  $\lvert\Out S\rvert
  =8f\le \log q^4\le \log {\ele(S)}$.

  Assume now that
  $n>4$, and that $q>2$ if $n$ is even. In this case, the smallest
  index maximal subgroups are the parabolic subgroups of type
  $[q^{2n-3}]:{\SU_{n-2}(q)}:((q^2-1)/d)$, with $d=\gcd(n,q+1)$. These
  subgroups have index
  ${\ele(S)}=(q^n-(-1)^n)(q^{n-1}-(-1)^{n-1})/(q^2-1)$. By
  \cite[Tables~8.20, 8.26, 8.37, 8.46, 8.56, 8.62, 8.72,
  and~8.78]{BrayHoltRoneyDougal13} and
  \cite[Table~3.5.B]{KleidmanLiebeck90}, we see that these maximal subgroups
  are ordinary. Since they have order
  \[\frac{1}{d}q^{n(n-1)/2}(q^2-1)\prod_{i=1}^{n-3}(q^{i+1}-(-1)^{i+1}),\]
  $(q^n+1)<q^2(q^{n-2}+1)$ and $q^{n-1}-1<q^{n-1}$ if $n$ is odd, and
  $(q^{n-1}+1)<q^2(q^{n-3}+1)$ and $q^n-1<q^n$ if $n$ is even, we see
  that the index of these subgroups is smaller than their
  order. Observe that
  \begin{align*}
    \ele(S)&=
    \begin{cases}
      (q^{n-1}-q^{n-2}+\dots -q+1)(q^{n-2}+q^{n-3}+\dots+q+1)&\text{($n$ odd)}\\
      (q^{n-1}+q^{n-2}+\dots +q+1)(q^{n-2}-q^{n-3}+\dots-q+1)&\text{($n$ even)}
    \end{cases}\\
    &=
      \begin{cases}
        ((q-1)(q^{n-2}+q^{n-4}+\dots+q)+1)(q^{n-2}+q^{n-3}+\dots+1)&\text{($n$ odd)}\\
        (q^{n-1}+q^{n-2}+\dots +1)((q-1)(q^{n-3}+q^{n-5}+\dots+q)+1)&\text{($n$ even)}        
      \end{cases}\\
    &\ge
      \begin{cases}
        q^{n-2}q^{n-2}&\text{($n$ odd)}\\
        q^{n-1}q^{n-3}&\text{($n$ even)}
      \end{cases}\\
    &=q^{2n-4}.
  \end{align*}
  It follows that $\log\ele(S)\ge (2n-4)\log q=(2n-4)f\log p$.
  Suppose that $p\ge 5$ or that $p=3$ and $n\ge 6$. Then $n\le (n-2)\log p$ and so
  \[\lvert \Out S\rvert \le n\cdot 2f \cdot 1\le (2n-4)f\log p\le \log \ele(S).\]
  Suppose now that $p=2$ and $n\ge 5$, or that $p=3$ and $n=5$, and $\gcd(n, q+1)=1$. In this case, $\lvert \Out S\rvert=2f$ and so
  \[\lvert \Out S\rvert \le 2f \cdot (2n-4)f\log 2\le \log \ele(S).\]
  Suppose now that $p=2$ and $n\ge 5$ or that $p=3$ and $n=5$, and that $\gcd(n, q+1)>1$. In this case, $n\le 3(n-2)\log p$ and so
  \[\lvert \Out S\rvert \le 2nf \le 3(2n-4)f\log p \le 3\log \ele(S).\]

  Finally, assume that $n\ge 6$, $n$ is even, and $q=2$. Then
  the smallest index maximal subgroups of $S$ have type
  $\SU_{n-1}(2):(3/d)$, with $d=\gcd(3,m)$, and index
  ${\ele(S)}=2^{n-1}(2^n-1)/3$,
  and, since $2^{n-1}\le 2^{(n-1)(n-2)/2}$ and
  $2^n-1<2^n+2=2(2^{n-1}+1)$, we obtain that ${\ele(S)}^2<\lvert
  S\rvert$. By \cite[Tables~8.26, 8.46, 8.62,
  and~8.78]{BrayHoltRoneyDougal13} and
  \cite[Table~3.5.B]{KleidmanLiebeck90}, we conclude that these
  maximal subgroups are ordinary. In this case, $\lvert
  \Out S\rvert=2\le \log {\ele(S)}$ if $n$ is not divisible by~$6$, while $\lvert \Out S\rvert = 3\cdot 2=6\le \log{\ele(S)}$ if $n$ is divisible by~$6$.

  \smalltitle{Orthogonal groups}

  Suppose now that $S\cong {\Orth_n^\varepsilon(q)}$ is an orthogonal group
  with $n\ge 7$, $n$ even if $q=2^f$. The smallest index maximal
  subgroups of $S$ have been described in
  \cite[Theorem]{VasilevMazurov94}. 

  Assume first that $n=8$, $\varepsilon={+}$ and $q>3$. Then we have
  that ${\ele(S)}=(q^4-1)(q^3+1)/(q-1)$ and we can take a maximal subgroup
  $H$ of type $q^6.({\Omega_6^+(q)}\times (q-1)/d).e$, where
  $d=\gcd(q^4-1,4)$, $e=\gcd(q^4-1, 2)$. Hence $\lvert
  H\rvert=q^{12}(q^3-1)(q^4-1)(q^2-1)(q-1)/e$. Since
  $(q^3-1)(q-1)>2$, we conclude that $q(q^3-1)=q^4-q>q^3+1$, and so
  ${\ele(S)}<\lvert H\rvert$.
  By \cite[Table~8.50]{BrayHoltRoneyDougal13},
  $S$ possesses an ordinary maximal subgroup in the Aschbacher class $\mathcal{C}_1$ of
  type $E_q^{1+8}(\frac{1}{e}{\GL_2(q)}\times
  \Omega_4^+(q)\bigr).e$ with index
  $v=(q+1)(q^2-q+1){(q^2+1)}^2(q^2+q+1)$.  Therefore 
  \[\frac{v}{{\ele(S)}^2}=\frac{q^3-1}{(q+1)^2(q^3+1)(q-1)}<1.\]
  It follows that $v<{\ele(S)}^2$. In this case, we have three conjugacy
  classes of maximal subgroups of index ${\ele(S)}$, fused under the
  triality outer automorphism.  Moreover, $\lvert \Out S\rvert =
  \gcd(2,q-1)^2\cdot f\cdot 3!$. If $p\ne 2$, then $2^8<3^6$ and so $\lvert \Out
  S\rvert \le 24f\le 3\log 2^{8f}\le 3\log p^{6f}\le
  3\log {\ele(S)}$. If $p\ge 17$, then $2^4\le p$ and so $\lvert \Out
  S\rvert \le 24f \le \log 2^{24f}\le \log p^{6f}\le \log {\ele(S)}$. If
  $p=2$, then $\lvert \Out S\rvert=6f\le \log q^6\le \log {\ele(S)}$.

  Suppose now that $S\cong {\Orth_n^+(2)}$ with $n=2t$ even. Then there is
  at least one 
  conjugacy class of maximal subgroups of smallest index of type $H\cong
  \Omega_{n-1}(2)$ and index ${\ele(S)}=2^{t-1}(2^t-1)$, and order
  \[\lvert H\rvert=2^{(n/2-1)^2}(2^{n/2}-1)(2^{n/2-1}-1)\dotsm
  (2^2-1);\]
  clearly ${\ele(S)}<\lvert H\rvert$. In this case, $\lvert \Out
  S\rvert \le 6\le \log {\ele(S)}$.

  By
  \cite{Atlas85}, if $n=8$, then there are three conjugacy classes of
  maximal subgroups of smallest index fused under the triality
  automorphism; moreover, ${\ele(S)}=120$ and $S$ has an ordinary maximal subgroup of type
  $2_+^{1+8}:(S_3\times S_3\times S_3)$ and index
  $1575<120^2={\ele(S)}^2$. Assume
  that $n\ge 10$. By \cite[Tables~8.66
  and~8.82]{BrayHoltRoneyDougal13} and
  \cite[Table~3.5.E]{KleidmanLiebeck90}, this subgroup is
  ordinary. Furthermore, $\lvert \Out S\rvert =6<\log {\ele(S)}$.

  Assume now that $S\cong {\Orth_n(3)}$, where $n=2t+1$ is odd. Then
there exists a conjugacy class of maximal
  subgroups isomorphic to $H\cong \Omega_{n-1}^-(3).2$, with index
  ${\ele(S)}=3^t(3^t-1)/2$ and
  \[\lvert
    H\rvert=3^{t(t-1)}(3^t+1)(3^{2t-2}-1)(3^{2t-4}-1)\dotsm
    (3^4-1)(3^2-1),\]
  so that ${\ele(S)}<\lvert H\rvert$. 
  Furthermore, by
  \cite[Tables~8.39, 8.58, and~8.74]{BrayHoltRoneyDougal13} and
  \cite[Table~3.5.D]{KleidmanLiebeck90}, these maximal subgroups are
  ordinary. In this case, $\lvert \Out S\rvert =2\cdot 1\cdot 2=4<\log {\ele(S)}$.

  Assume that $S\cong {\Orth_n^+(3)}$, where $n=2t$ is even. The maximal
  subgroups of the smallest index have type $H\cong
  \Omega_{n-1}(3).h$, where $h=\gcd(t-1,2)$. Their index is
  ${\ele(S)}=3^{t-1}(3^t-1)/2$, clearly smaller than
  their order $\lvert
  H\rvert=(1/2)3^{(t-1)^2}\prod_{i=1}^{t-1}(3^{2i}-1)$. In this case,
  $\lvert \Out S\rvert =2^2\cdot 1\cdot 2=8< \log {\ele(S)}$ if $n\ge 10$.
 
  In the case that $S\cong {\Orth_8^+(3)}$, ${\ele(S)}=1\,080$, there are six
  conjugacy classes of maximal subgroups isomorphic to $H$, the outer automorphism
  group of $S$ is isomorphic to $S_4$ and there is an ordinary subgroup of
  type $3_+^{1+8}:2({\Alt(4)}\times \Alt(4)\times \Alt(4)).2$, according to
  \cite{Atlas85}, of index $36\,400<{\ele(S)}^2$.  Moreover,  $\lvert \Out S\rvert = 24 <3\log 1\,080=3\log {\ele(S)}$.

  Assume that $n=2t\ge
  10$. By \cite[Tables~8.66 and 8.82]{BrayHoltRoneyDougal13} and
  \cite[Table~3.5.E]{KleidmanLiebeck90}, there are two conjugacy
  classes of subgroups of the smallest index $\ele(S)=3^t(3^t-1)/2$. On the other hand, there is an ordinary maximal parabolic
  subgroup of type $3^{n-2}.(\Omega^+_{n-2}(3).2)$, of index
  $(3^t-1)(3^{t-1}+1)/2<{\ele(S)}^2$. In this case, $\lvert \Out S\rvert = \gcd(2, 3-1)^2\cdot 1 \cdot 2=8$ and, since $\ele(S)\ge 3^5(3^5-1)/2>2^8$, we conclude that the inequality $\lvert \Out S\rvert < \log\ele(S)$ also holds in this case.

  Suppose that $S\cong {\Orth_n(q)}$ where $n=2t+1$ is odd, $q=p^f\ne 3$,
  and $p$ is an odd prime. Then the smallest index of a maximal
  subgroup of $S$ corresponds to $H\cong
  [q^{n-2}].((\Omega_{n-2}(q)\times (q-1)/2).2)$, of index
  $(q^{n-1}-1)/(q-1)$. Since $\lvert H\rvert =
  (1/2)q^{n-2}q^{(t-1)^2}\prod_{i=1}^{t-1}(q^{2i}-1)$, it is clear
  that $\lvert H\rvert > {\ele(S)}$.
  By \cite[Tables~8.39, 8.58,
  and~8.74]{BrayHoltRoneyDougal13} and
  \cite[Table~3.5.D]{KleidmanLiebeck90}, these maximal subgroups are
  ordinary. Moreover, $\lvert \Out S\rvert = 2^2\cdot f\cdot 1\le
  2\log 2^f\le 2\log q <\log
  {\ele(S)}$.

  Suppose now that $n=2t$, $q=2^f$, $f\ge 2$, and that $(n,\varepsilon)\ne
  (8,{+})$. In this case, the smallest index of a maximal subgroup
  corresponds to subgroups of type
  $H=[q^{n-2}].(\Omega_{n-2}^\varepsilon(q)\times (q-1))$, of index
  ${\ele(S)}=(q^t-\varepsilon)(q^{t+1}+\varepsilon)/(q-1)$  and order
  $\lvert H\rvert =q^{n-2}q^{(t-1)(t-2)}(q^{t-1}-\varepsilon)\prod_{i=1}^{t-1}(q^{2i}-1)$. We
  can see that $\lvert H\rvert > {\ele(S)}$.
  By \cite[Tables~8.52, 8.66,
  8.68, 8.82, and 8.84]{BrayHoltRoneyDougal13} and \cite[Tables~3.5.E
  and~3.5.F]{KleidmanLiebeck90}, we see that these subgroups are
  ordinary. Furthermore, $\lvert \Out S\rvert=f\cdot 2$ if
  $\varepsilon={+}$ and $\lvert \Out S\rvert= 2f\cdot 1$ if
  $\varepsilon={-}$. In both cases, $\lvert \Out S\rvert \le \log
  {\ele(S)}$.

  Finally, suppose that $n=2t$, $q=p^f$, $p$ is an odd prime, the pair
  $(m,\varepsilon)$ is different from $(8,{+})$, and $(q,\varepsilon)$
  is different from $(3, {+})$. Then the smallest index of a maximal
  subgroup corresponds to the subgroups of type $H\cong
  p^{f(m-2)}.((\Omega_{m-2}^\varepsilon(q)\times (q-1)/h).j$, where
  $(h,j)=(2,2)$ if $\gcd(q^t-\varepsilon, 4)=2$, $(h,j)=(2,1)$ if
  $\gcd(q^t-\varepsilon, 4)=4$ and $\varepsilon(q^{t-1}-\varepsilon,
  4)=2$, and $(h,j)=(4,2)$ if $\gcd(q^t-\varepsilon,4)=4$ and
  $\gcd(q^{t-1}-\varepsilon, 4)=1$. This index is
  ${\ele(S)}=(q^t-\varepsilon)(q^{t-1}+\varepsilon)/(q-1)$. As in the
  previous cases, ${\ele(S)}< \lvert H\rvert$. 
  This subgroup is
  ordinary by \cite[Tables~8.52, 8.66, 8.68, 8.82,
  and 8.84]{BrayHoltRoneyDougal13} and \cite[Tables~3.5.E
  and~3.5.F]{KleidmanLiebeck90}. Furthermore, $\lvert \Out S\rvert \le
  2^2\cdot f\cdot 2\le \log q^6\le \log {\ele(S)}$.

\smalltitle{Groups of type~$G_2(q)$}

The maximal subgroups of smallest index of the simple groups $G_2(q)$,
$q>2$, have been studied in \cite[Theorem~1]{Vasilyev96}.

Assume that $S\cong G_2(3)$. Then $P\cong \PSU_3(3):2$ is a maximal
subgroup of the smallest possible index ${\ele(S)}=351$  and order
$12\,096>{\ele(S)}$. By \cite{Atlas85}, there are two conjugacy classes of
maximal subgroups of this index and there is a conjugacy class of
ordinary maximal subgroups of type $\PSL_2(8):3$ and index
$2\,808<{\ele(S)}^2$. In this case, $\lvert \Out S\rvert =2<\log {\ele(S)}$.

Assume that $S\cong G_2(4)$. Then $P\cong J_2$ is a maximal subgroup
of the smallest possible index ${\ele(S)}=416$  and order
$604\,800>{\ele(S)}$. According to \cite{Atlas85}, this subgroup is
ordinary. Moreover, $\lvert \Out S\rvert = 2<\log {\ele(S)}$.

Suppose now that $S\cong G_2(q)$ with $q\ge 5$. Then
${\ele(S)}=(q^6-1)/(q-1)$.

Assume that $S\cong G_2(q)$ with $q=2^f$, $f\ge 3$. Then $P_1\cong (2^f
. 2^{4f}):({\PSL_2(q)}\times (q-1))$ and $P_2\cong
(2^{2f}. 2^{3f}):({\PSL_2(q)}\times (q-1))$  are  maximal subgroups of $S$ of
the smallest possible index. We see in
\cite[Table~8.30]{BrayHoltRoneyDougal13} that these subgroups are
ordinary. Clearly, ${\ele(S)}<\lvert P\rvert$  and $\lvert \Out
S\rvert=f\le \log {\ele(S)}$.

Assume that $S\cong G_2(3^f)$ with $f\ge 2$. Then the smallest index
maximal subgroups of $S$ are of type $P\cong (3^f . 3^{2f}\times
3^{2f}):(2.({\PSL_2(q)}\times(q-1)/2).2)$. Note that $\lvert P\rvert >
{\ele(S)}$.
By
\cite[Table~8.42]{BrayHoltRoneyDougal13}, there are two conjugacy
classes of subgroups of this type. There is a conjugacy class of
ordinary subgroups of type
$({\SL_2(q)}\circ {\SL_2(q)}).2$ and index
$q^4(q^4+q^2+1)=q^4(q^6-1)/(q^2-1)<{\ele(S)}^2$. Moreover, $\lvert \Out
T\rvert=2f\le \log {\ele(S)}$.

Now assume that $q=p^s$, with $p$ a prime, $p>3$. Then there are two
conjugacy classes of maximal subgroups of the smallest index, namely
$P_1\cong (p^{2f}.(p^f . p^{2f})):(2.({\PSL_2(q)}\times (q-1)/2).2)$
and $P_2\cong (p^f . p^{4f}):(2.({\PSL_2(q)}\times
(q-1)/2).2)$. Again, $\lvert P_1\rvert > {\ele(S)}$, $\lvert \Out S\rvert =
f\le \log {\ele(S)}$ 
and,
by
\cite[Table~8.41]{BrayHoltRoneyDougal13}, these subgroups are
ordinary.

\smalltitle{Groups of type $F_4(q)$}

The smallest index maximal subgroups of $F_4(q)$ have been studied in
\cite[Theorem~2]{Vasilyev96}. This index is
\[{\ele(S)}=\frac{(q^{12}-1)(q^4+1)}{q-1}\]
and is attained by a parabolic subgroup. 
Since the order of this subgroup is
$\lvert P\rvert =q^{24}(q^4-1)(q^6-1)(q^2-1)(q-1)$, we have that
$\lvert P\rvert > {\ele(S)}$. Moreover, $\lvert \Out S\rvert$ is $f$ if
$p\ne 2$ and $2f$ if $p=2$. In both cases, $\lvert \Out S\rvert \le
\log q^4\le \log {\ele(S)}$.

Assume first that $q=2^f$. Then there are two conjugacy classes of
parabolic maximal subgroup isomorphic to
$P\cong (2^f. 2^{8f}\times 2^{6f}):({\PSp_6(q)}\times (q-1))$. 
By \cite[Table~5.1]{LiebeckSaxlSeitz92-plms},
$S$ has an ordinary maximal subgroup of
type $H\cong e.\bigl(L_3^\varepsilon(q)\times
L_3^\varepsilon(q)\bigr).e.2$, where $\varepsilon=\pm 1$,
$e=\gcd(3,q-\varepsilon)$, $L_3^{+1}(q)=\PSL_3(q)$ and
$L_3^{-1}(q)=\PSU_3(q)$. Then $v=\lvert S:H\rvert
=q^{18}(q+1)^2(q^2-q+1)^2(q^2+1)^2(q^4-q^2+1)(q^4+1)$  and we can check
that $v\le {\ele(S)}^2$. 

Assume that $q=p^f$ with $p$ a prime different from $2$. Then
$P_1=(p^f.p^{14f}):(2.({\PSp_4(q)}\times (q-1)/2).2)$ or
$P_4=(p^{7f}.p^{8f}):(2.({\Orth_7(q)}\times (q-1)/2).2)$ are parabolic maximal
subgroups of $S$ of the smallest index. The conjugacy classes of both
subgroups are fixed under the outer automorphism group of $S$ since
both are parabolic.

\smalltitle{Groups of type $E_6(q)$}
  The maximal subgroups of smallest index of $S\cong E_6(q)$ have been
  studied in \cite[Theorem~1]{Vasilyev97}. The smallest index of a
  maximal subgroup of $S$ is
  ${\ele(S)}=(q^9-1)(q^8+q^4+1)/(q-1)=(q^9-1)(q^{12}-1)/\bigl((q^4-1)(q-1)\bigr)$,
  corresponding to two conjugacy
  classes of parabolic
  subgroups $P_1=p^{16f}:(e.{\Orth_{10}^+(q)}\times (q-1)/e').e)$, where
  $e=\gcd(q-1,4)$, $e'=e\cdot\gcd(q-1,3)$, interchanged by the graph
  automorphism.  Clearly, ${\ele(S)}<\lvert P_1\rvert$. Moreover, $\lvert
  \Out S\rvert =\gcd(3,q-1)\cdot f\cdot 2\le 6f\le \log q^8\le \log
  {\ele(S)}$.
  The parabolic
  subgroup $P_2$ is ordinary and is of 
  type $[q^{21}]:H$ where $H$ has a section isomorphic to $\PSL_6(q)$ (see also \cite[Table~7.3]{Craven21-arXiv-max-sgps-except})
  and so its index divides
  $v=(q^{12}-1)(q^4+1)(q^9-1)/(q^3-1)$. Then
  \[\frac{v}{{\ele(S)}^2}=\frac{(q^8-1)(q^4-1)(q-1)^2}{(q^{12}-1)(q^9-1)(q^3-1)}<1,\]
  therefore $v< {\ele(S)}^2$. We conclude that $\lvert S:P_2\rvert \le v  \le {\ele(S)}^2$.

\smalltitle{Groups of type $E_7(q)$}

  The maximal subgroups of smallest index of $S\cong E_7(q)$ for
  $q=p^f$ have been
  studied in \cite[Theorem~2]{Vasilyev97}. They are the parabolic
  subgroups $P\cong p^{27f}:(d'.(E_6(q)\times (q-1)/c).d')$, with
  $d'=\gcd(q-1,3)$, $c=\gcd(q-1,2)\cdot d'$, of index
  ${\ele(S)}=(q^{14}-1)(q^9+1)(q^5+1)/(q-1)$. Clearly $P_1$ is ordinary,
  ${\ele(S)}<\lvert P_1\rvert$, and $\lvert \Out S\rvert=\gcd(2,q-1)\cdot
  f\cdot 1\le 2f\le \log q^5\le \log {\ele(S)}$.

\smalltitle{Groups of type $E_8(q)$}

  The maximal subgroups of smallest index of $S\cong E_8(q)$ for
  $q=p^f$ have been
  studied in \cite[Theorem~3]{Vasilyev97}. They are the parabolic
  subgroups $P\cong (p^f.p^{56f}):(d.(E_7(q)\times (q-1)/d).d)$, with
  $d=\gcd(q-1,2)$, of index
  ${\ele(S)}=(q^{20}-1)(q^{12}+1)(q^{10}+1)(q^6+1)/(q-1)$. Clearly $P$ is
  ordinary and ${\ele(S)}<\lvert P_1\rvert$, and $\lvert \Out T\rvert=f\le
  \log q^6\le \log {\ele(S)}$.

\smalltitle{Twisted groups}
In \cite[Theorem~1]{Vasilyev98}, it is shown that if $S\cong
{}^2\!B_2(q)$, with $q=2^f$, $f$ an odd integer greater than $1$, the
smallest index of a maximal subgroup of $S$ corresponds to the
parabolic subgroup $P\cong (2^f.2^f):(q-1)$, with index
${\ele(S)}=q^2+1$. By \cite[Table~8.16]{BrayHoltRoneyDougal13}, these
subgroups are ordinary and, clearly, ${\ele(S)}< \lvert P\rvert$  and
$\lvert \Out S\rvert = f= \log q\le \log {\ele(S)}$. 
 
In \cite[Theorem~2]{Vasilyev98}, it is shown that if $S\cong
{}^2\!G_2(q)$, with $q=3^f$ and $f$ an odd integer greater than~$1$,
there is a class of smallest index maximal subgroups isomorphic to
$P\cong (3^f.3^f.3^f):(q-1)$ and index $q^3+1$. By
\cite[Table~8.43]{BrayHoltRoneyDougal13}, these subgroups are ordinary
and, clearly, ${\ele(S)}< \lvert P\rvert$ and $\lvert \Out S\rvert = f\le
\log q\le \log {\ele(S)}$.

In \cite[Theorem~3]{Vasilyev98}, it is shown that if $S\cong
{}^3\!D_4(q)$, with $q=p^f$, the smallest index maximal subgroups of $S$ are
isomorphic to $P\cong (p^f.p^{8f}):(d.({\PSL_2(q^3)}\times
(q-1)/d).d)$, where $d=\gcd(2,q-1)$, with index
${\ele(S)}=(q^8+q^4+1)(q+1)$. By \cite[Table~8.51]{BrayHoltRoneyDougal13},
these subgroups are ordinary. Moreover, $\lvert P\rvert
=dq^{12}(q^6-1)(q-1)=dq^{12}(q-1)^2(q+1)(q^4+q^2+1)>{\ele(S)}$ and $\lvert
 \Out S\rvert =f\le \log q\le \log {\ele(S)}$.

In \cite[Theorem~4]{Vasilyev98}, it is shown that for $S\cong
{}^2\!E_6(q)$, with $q=p^f$, the smallest index maximal subgroups of $S$ are
isomorphic to $P\cong (p^f.p^{20f}):(d_+.{\PSU_6(q)}\times
(q-1)/d'_+).d'_+$, where $d_+=\gcd(2,q+1)$, $d_+'=\gcd(3,q+1)$. Their
index is ${\ele(S)}=(q^{12}-1)(q^6-q^3+1)(q^4+1)/(q-1)$. These subgroups
are clearly ordinary because they are parabolic. Clearly, $\lvert
P\rvert > {\ele(S)}$  and $\lvert \Out S\rvert =\gcd(3,q+1)\cdot f\cdot 1\le
3f\le \log q^{11}\le \log {\ele(S)}$. 

By \cite[Theorem~5]{Vasilyev98}, if  $S\cong {}^2\!F_4(q)$, with
$q=2^f$, $f>1$ odd, the smallest index maximal subgroups of $S$ are
isomorphic to $P\cong (2^f.2^{4f}.2^{5f}):({}^2\!B_2(q)\times
(q-1))$, with index ${\ele(S)}=(q^6+1)(q^3+1)(q+1)
$ and order
$\lvert P\rvert = q^{12}(q^2+1)(q-1)^2$.  It is clear that $\lvert
P\rvert < {\ele(S)}$ and $\lvert \Out S\rvert = f\le \log q\le \log {\ele(S)}$. 
Moreover, the subgroup $P$ is ordinary because it is
a parabolic subgroup. 
\end{proof}
\begin{remark}
  We thank one of the anonymous referees for drawing our attention to the interesting paper \cite{AlaviBurness15} of Alavi and Burness. These authors have obtained in their Theorems~2--5 for each simple group $G$ and in their Theorem~7 for each almost simple group~$G$ the list of all maximal subgroups $H$ of $G$ with $\lvert H\rvert^3 \ge \lvert G\rvert$. They call them \emph{large}. In fact, all maximal subgroups appearing in the proof of Theorem~\ref{th-simple} are large in this sense and so all of them are mentioned in~\cite{AlaviBurness15}.
\end{remark}
\begin{remark}
  Note that the smallest index of a smallest core-free maximal subgroup of an almost
  simple group with socle $\PSL_n(q)$ with $n\ge 3$ can be different
  from the indices of the parabolic and the double parabolic
  subgroups. According to
  \cite{Atlas85}, if $S=\PSL_3(4)$, then the extension $S.2_1$
  contains a maximal subgroup of type $\mathrm{M}_{10}$ and least index $56$, different from the indices of the
  parabolic subgroups of type $P_1$, of index $21$, that do not exist
  in this extension, and the double parabolic
  subgroups of type $P_{1,2}$, of index $105$, that also appear as a
  maximal subgroup of $S.2_1$.
\end{remark}

\begin{remark}\label{remark-PSLn2}
  Let $S=\PSL_n(2)\cong \GL_n(2)$, where $n$ is a prime, $n\ge 5$. There is a
  unique class of ordinary maximal subgroups of $S$ of geometric type by Tables~8.18, 8.19, 8.36,
  8.37, 8.70, 8.71 of \cite{BrayHoltRoneyDougal13} and
  \cite[Table~3.5.A]{KleidmanLiebeck90}, namely the subgroup
  $M=\GL_1(2^n):n$ in the Aschbacher class $\mathcal{C}_3$. Note that
  $\lvert S\rvert=(2^n-1)(2^n-2)\dotsm (2^n-2^{n-1})$, while $\lvert
  M\rvert=(2^n-1)n$. Consequently, $\lvert S:M\rvert=(2^n-2)\dotsm
  (2^n-2^{n-1})/n$. The smallest index of a core-free maximal subgroup
  of $S$ is smaller or equal than the index of the parabolic subgroup
  $P_1$, corresponding to the stabiliser of a vector subspace of
  dimension $1$. Since $\lvert
  P_1\rvert=2^{n-1}(2^{n-1}-1)(2^{n-1}-2)\dotsm (2^{n-1}-2^{n-2})$, we
  have that $\lvert S:P_1\rvert=2^n-1$. Now the largest power of~$2$ dividing
  $\lvert S:M\rvert$ is $2\cdot 2^2\dotsm
  2^{n-1}=2^{n(n-1)/2}$. Therefore $\lvert
  S:M\rvert={(2^n)}^{(n-1)/2}>{(2^n-1)}^{(n-1)/2}\ge
  {\ele(S)}^{(n-1)/2}$. In particular, there cannot exist a constant $c$
  such that if $\ele_1(S)$ is the common index of a maximal subgroup of
  geometric type in
  all almost simple groups associated with the non-abelian simple
  group $S$, $\ele_1(S)\le {\ele(S)}^c$ for all non-abelian simple
  groups~$S$.
\end{remark}
\begin{remark}
  The groups $\PSL_n(q)$ for $n\ge 3$, $q=q_0^2$, $q_0$ a prime power, contain always a maximal subgroup of the form $\PSL_n(q_0)$ or $\PSU_n(q_0)$, but their indices in $\PSL_n(q)$ are polynomials on $q_0$ of degree larger than  the degree of ${\ele({\PSL_n(q)})}^2=(q_0^{2n}-1)^2/(q_0^2-1)^2$ when $n\ge 4$. Hence this construction cannot be extended further. The groups $\PSL_{2m}(q)$ for $m\ge 2$, $q$ a prime power, $(2m,q)\ne (4,2)$, contain a parabolic subgroup $P_{m}$ that is ordinary. However, for $m\ge 4$, this subgroup has index in $\PSL_{2m}(q)$ that is a polynomial on $q$ of degree larger than the one of  ${\ele({\PSL_{2m}(q)})}^2=(q^{2m}-1)^2/(q-1)^2$. This justifies that these constructions cannot be extended further and so $\PSL_3(q_0^2)$, $\PSL_4(q)$, $\PSL_6(q)$ belong to the class $\mathfrak{Y}$, but not the linear groups in larger dimensions.
\end{remark}

\begin{remark}\label{remark-PSL-log}
  Let $S=\PSL_{m}(2^f)$ with $m\ge 3$ and $m\mid 2^f-1$ (for example,
  $m-1\mid f$). Then $\lvert {\Out S }\rvert =
  (2^f-1)\cdot f\cdot 2$ and ${\ele(S)}=(2^{mf}-1)/(2^f-1)$. Since
  \[\lim_f\frac{\log ((2^{mf}-1)/(2^f-1))}{m\cdot
    f\cdot 2}=\frac{m-1}{2m},\]
  it follows that the 
  bound $\lvert \Out S\rvert\le 3\log {\ele(S)}$ cannot be improved.
\end{remark}

\begin{proof}[Proof of Theorem~\ref{th-s-log3}]
  The result is clear for sporadic and alternating groups, since then
  the outer automorphism group is trivial, isomorphic to $C_2$, or
  isomorphic to $C_2\times C_2$. It only remains to consider the case
  when $S$ is a simple group of
  Lie type. According to \cite[Table~5]{Atlas85}, the outer
  automorphism group is isomorphic to an extension of a metacyclic
  group by a cyclic group, with the possible exception of
  $\Orth_{2m}^+(q)$ with $m\ge 4$, $m$ even. By \cite[page~181]{Kleidman87}, 
  \[\Out({\Orth_8^+(q)})\cong
  \begin{cases}
    \Sym(3)\times C_f&\text{if $q$ is even,}\\
    \Sym(4)\times C_f&\text{if $q$ is odd.}
  \end{cases}
  \]
  If $m\ge 6$ is even, the same arguments show that
  \[\Out({\Orth_{2m}^+(q)})\cong 
  \begin{cases}
    D_8\times C_f&\text{if $q$ is even,}\\
    C_2\times C_f&\text{if $q$ is odd.}
  \end{cases}
  \]
  By considering the normal subgroup isomorphic to $C_f$, we see that
  all subgroups of $\Out({\Orth_{2m}^+(q)})$ for $m\ge 4$, $m$ even, can also be
  generated by at most $3$ generators. In all other cases, all
  subgroups of $\Out S$ are extensions of a metacyclic group by a
  cyclic group and so they are also $3$-generated.

  If $\lvert \Out S\rvert \le \log {\ele(S)}$, then given a subgroup of
  $G$, we have at most $\log^3{\ele(S)}$ possibilities for a generating
  set. It follows that the number of subgroups of $\Out S$ is at most
  $\log^3 {\ele(S)}$. Therefore we must study the cases in which the
  inequality $\lvert \Out S\rvert \le \log {\ele(S)}$ can fail, that is,
  the cases mentioned in
  Theorem~\ref{th-simple}~(\ref{en-lemma-simple-1log}).

  We begin with the linear groups $S=\PSL_m(q)$. Suppose first that
  $m=2$ and that $q=3^f$. Then 
  \[\lvert \Out S\rvert = 2\cdot f\cdot 1=\frac{2}{\log 3}\log 3^f\le
  \frac{2}{\log 3}\log {\ele(S)}.\]
  Note that $\Out S$ and all its subgroups are
  $2$-generated. One of the generators can be taken in the subgroup of
  order~$2$, while the other one can be chosen in $2f$ different
  ways. This gives for the number of subgroups an upper bound of
  \[4f\le 2\cdot\frac{2}{\log 3}\log {\ele(S)}=\frac{4}{\log 3}\log {\ele(S)}.\]
  Since $4/{\log 3}\le \log^25 \le \log^2{\ele(S)}$,  we conclude that the number of subgroups of $\Out
  S$ is bounded by  $\log^3{\ele(S)}$.

  Let $S=\PSL_m(q)$ with $m\ge 3$. Then
  \begin{align*}
    \lvert \Out S\rvert &=\gcd(m,q-1)\cdot f\cdot 2\le
    2mf\\&=\frac{2m}{(m-1)\log p}\log p^{(m-1)f}\le \frac{2m}{(m-1)\log
      p}\log {\ele(S)}.
  \end{align*}

Note that
  \[\frac{2m}{(m-1)\log p}\le 3.\]
  
  Suppose that $f=1$.   Since $\ele({\PSL_3(2)})=7$ and $\lvert{\Out{\PSL_3(2)}}\rvert=2$, and $\ele({\PSL_4(2)})=8$ and $\lvert{\Out{\PSL_4(2)}}\rvert=2$, we can assume that $(m,q)\notin\{(3,2), (4,2)\}$. Consequently, $\ele(S)=(q^m-1)/(q-1)=q^{m-1}+q^{m-2}+\dots+q+1$ and so $\log\ele(S)>(m-1)\log q$.
We have that $\Out S$ is $2$-generated and has order
  $\gcd(m,q-1)\cdot 2$. Moreover, all subgroups of $\Out S$ are
  $2$-generated, the first generator can be taken in the normal cyclic
  subgroup of order $\gcd(m,q-1)$ and the second one in $\Out S$. This
  gives at most $\gcd(m,q-1)^2 \cdot 2$
  possibilities for a subgroup of $\Out S$. Since $(m-1)^3\log^3q-2m^2>0$ for $m\ge 3$ if $q\ge 3$ and $(m-1)^3\log^3 2-2m^2=(m-1)^3-2m^2>0$ for $m\ge 5$, we have that
  \[\gcd(m, q-1)^2\cdot 2\le 2m^2\le (m-1)^3\log^3q\le \log^3\ele(S)\]
  for $m\ge 3$, $(m,q)\notin \{(3,2), (4,2)\}$.
  
  Hence we can assume that $f\ge 2$. Every subgroup of $\Out S$ is
  $3$-generated, and the generators can be taken one in the group of
  diagonal automorphisms of order $\gcd(m,q-1)$, another one in the group
  generated by the diagonal and field automorphisms of order
  $\gcd(m,q-1)f$, and the third one in $\Out S$. It follows that the
  number of possible choices is at most $2f^2\gcd(m,q-1)^3$. Suppose
  first that $\gcd(m,q-1)<m$, then $\gcd(m,q-1)\le m/2$ and so the
  number of possible subgroups is bounded by 
  \[2f^2\frac{m^3}{8}\le \frac{1}{8}\frac{(2fm)^3}{8}\le
  \frac{1}{8}\frac{m^3}{(m-1)^3\log^3p}\log^3{\ele(S)}<\log^3{\ele(S)}.\]
  Therefore we can assume that $\gcd(m,q-1)=m$, that is, $m\mid
  q-1$. The number of possible subgroups of $\Out S$ is bounded by
  \[2f^2m^3\le \frac{(2fm)^3}{8}\le
  \frac{m^3}{(m-1)^3\log^3p}\log^3{\ele(S)}.\]
  If $p\ge 3$, then $m^3<(m-1)^3\log^3p$ and so the number of possible
  subgroups of $\Out S$ is again bounded by $\log^3{\ele(S)}$. Therefore we
  can suppose that $p=2$. In particular, $m$ must be odd. Assume that $m$ is a
  prime. Then the number of choices of the element of the group of
  diagonal automorphisms can be reduced from $m$ to $2$, namely the trivial
  element and a generator. This gives that the number of possible
  subgroups of $\Out S$ is bounded by 
  \[2\cdot mf\cdot 2mf=(2mf)^2\le \frac{4m^2}{(m-1)^2}\log^2{\ele(S)}.\]
  If $m=5$, then ${\ele(S)}\ge (2^4)^4=2^{16}$, and so $\log {\ele(S)}\ge
  16$. It follows that $4m^2/(m-1)^2<\log {\ele(S)}$. If $m=7$, then
  ${\ele(S)}\ge (2^3)^6=2^{18}$ and so $\log {\ele(S)}\ge 18$. Consequently,
  $4m^2/(m-1)^2<\log {\ele(S)}$. Hence for $m\in\{5,7\}$, the number of
  subgroups of $\Out S$ is bounded by $\log^3{\ele(S)}$. Assume now
  that $m\ge 9$ is odd. The number of choices of the element of the
  group of diagonal automorphisms can be reduced to the number of
  subgroups of this cyclic group, which coincides with the number of 
  divisors of $m$. Since $m$ is odd, this number is not greater than
  $2m/3$. The number of possible choices for the generators of a
  subgroup of $\Out S$ is bounded by
  \[(2m/3)\cdot mf\cdot 2mf=\frac{1}{6f}(2mf)^3\le
  \frac{1}{12}\frac{8m^3}{(m-1)^3}\log^3 {\ele(S)},\]
  and $8m^3/\bigl(12(m-1)^3\bigr)\le 243/256$, so that this number is
  bounded by $\log^3{\ele(S)}$. It only remains the case $m=3$. In this
  case, the group of outer automorphisms of $S$ has the presentation
  \[\Out S= \langle x, y, z\mid x^3=y^f=z^2=1, x^y=x^{-1},
  x^z=x^{-1}, y^z=y^{-1}\rangle.\]
  Note that $\langle y^2\rangle$ centralises $\langle x\rangle$. Now
  \[\Out S/\langle y^2\rangle\cong \langle a,b \mid a^3=b^2=1,
  a^b=a^{-1}\rangle\times \langle c\mid c^2=1\rangle\cong
  \Sym(3)\times C_2,\]
  where $a=\bar x$, $b=\bar z$, $c=\bar y \bar z$. Note that every
  subgroup of $\Sym(3)\times C_2$ is $2$-generated. A pair of
  generators can be obtained by taking an element of $\langle
  a,b\rangle$ and an element of the set $\{1,c,ac,bc,abc,a^2bc\}$. The
  preimages of these sets under the natural epimorphism from $\Out S$
  onto $(\Out
  S)/\langle y^2\rangle$ have $6(f/2)=3f$ elements each. Hence every
  element of $\Out S$ can be obtained by considering an element of
  $\langle y^2\rangle$, for which we have $f/2$ choices, and the
  $3f$ choices for each element of the preimages. This gives a
  bound for the number of subgroups of
  $(f/2)(3f)^2=9f^3/2=9(2f)^3/16<(9/16)\log^3{\ele(S)}<\log^3{\ele(S)}$. This
  completes the proof for the linear case.

  If $S\cong \PSU_3(5)$, then $\Out S\cong \Sym(3)$ has $6$
  subgroups, clearly $6\le \log^350=\log^3{\ele(S)}$ by
  \cite{Atlas85}. Suppose that $S\cong \PSU_3(q)$ with
  $q\notin\{2,5\}$. Then ${\ele(S)}=q^3+1$ by \cite[Theorem~3]{Mazurov93} and $\Out S$ is a metacyclic
  group of order $3\cdot 2f=6f$ and all its subgroups are also
  $2$-generated. Then the number of choices for a
  couple of generators for a subgroup of $G$ is bounded by $3f\cdot
  6f\le (2/3)\log {\ele(S)} \cdot (4/3)\log {\ele(S)}<\log^2{\ele(S)}< \log^3{\ele(S)}$. 
  Assume now that $S\cong \PSU_m(q)$ with $q>2$ if $m$ is even and
  $q^2=p^f$, with $p$ a prime. By \cite[Theorem~3]{Mazurov93},
  ${\ele(S)}=(q^m-(-1)^m)(q^{m-1}-(-1)^{m-1})/(q^2-1)$. Moreover,
  $\lvert\Out S\rvert=m\cdot (2f)\cdot 1$ is $2$-generated. Since
  $2mf\le 3\log {\ele(S)}^2$, we have that the number of choices for the
  pair of generators of a subgroup of $\Out S$ is bounded by $mf\cdot
  2mf\le (9/2)\log^2{\ele(S)}\le \log^3{\ele(S)}$, because $9/2\le \log
  28=\log^3{\ele \bigl({\PSU_3(3)}\bigr)}<\log {\ele(S)}$ by \cite{Atlas85}.

  Finally, suppose that $G\cong {\Orth_8^+(q)}$, where $q=p^f$ and
  $p\in\{3,5,7,11,13\}$. Now ${\ele(S)}=(q^3+q^2+q+1)(q^3+1)$ by
  \cite[Theorem]{VasilevMazurov94} and $\Out S\cong \Sym(4)\times
  C_f$. All subgroups of $\Out S$ are $3$-generated, and one generator
  can be taken in $C_f$, another one in $\Alt(4)\times C_f$ and the other
  one in the whole group. This gives at most $f\cdot (12f)\cdot (24f)=288f^3$
  possibilities for a subgroup of $\Out S$. Since $24f\le 3\log {\ele(S)}$,
  we have that $8f\le \log {\ele(S)}$. Consequently, $288f^3\le 512f^3\le
  \log^3{\ele(S)}$. This completes the proof of the theorem.
\end{proof}
\begin{remark}
  We note that in the Janko sporadic group $S\cong
  J_3$
  of order $50\,232\,960$ and with ${\ele(S)}=6\,156$, if $\xi=\log 50\,232\,960/\log 6\,156\approx 2.0323$, we have that $\lvert S\rvert=\ele(S)^\xi$. Therefore the exponent~$2$ in $\ele(S)^2\le \lvert S\rvert$ cannot be increased too much.
\end{remark}

\backmatter

\section*{Acknowledgement}
Open Access funding provided thanks to the CRUE-CSIC agreement with Springer Nature.
We thank the anonymous referees for their huge effort in reading carefully the manuscript and for identifying several typos and incorrect statements. Their contributions have improved considerably the presentation of the results of this paper and their proofs. We also thank the Editor-in-Chief, Professor Fernando Etayo, for his kind help in the revision of the manuscript and for his empathy and his patience with our concerns.

\section*{Declarations}

\subsection*{Funding}
These results are part of the R+D+i project supported by the Grant 
PGC2018-095140-B-I00, funded by MCIN/AEI/10.13039/501100011033 and by ``ERDF A way of making Europe'', as well as by the Grant PROMETEO/2017/057
funded by GVA/10.13039/501100003359, and partially supported by the Grant E22 20R, funded by Departamento de Ciencia, Universidades y Sociedad del
Conocimiento, Gobierno de Arag\'on/10.13039/501100010067.

\subsection*{Conflict of interest/Competing interests}
The authors have no relevant financial or non-financial interests to disclose.

\subsection*{Consent to participate}
Not applicable
\subsection*{Consent for publication}
Not applicable
\subsection*{Availability of data and materials}
Not applicable
\subsection*{Code availability}
Not applicable

\subsection*{Authors' contributions}
All authors have contributed equally to this paper.


\providecommand{\inpress}{\iflanguage{spanish}{en
 prensa}{\iflanguage{catalan}{en premsa}{in press}}}

\end{document}